\documentclass[a4paper,oneside,portrait,12pt]{AmsArt}

\usepackage[margin=1in]{geometry}

\usepackage[ps,all,arc,rotate]{xy}

\usepackage {graphicx}
\usepackage{amsfonts}
\usepackage{amsthm}
\usepackage{amsmath}
\usepackage{amsfonts}
\usepackage{latexsym}
\usepackage{amssymb}
\usepackage{epsfig,color}
\usepackage{graphicx, amssymb}

\usepackage{hyperref}
\hypersetup{
  bookmarks=true,
}

\newtheorem{definition}[]{Definition} 
\newtheorem{theorem}[definition]{Theorem}

\newtheorem{proposition}[definition]{Proposition}
\newtheorem{remark}[definition]{Remark}

\theoremstyle{definition}

\def\CM{{\mathcal M}}

\title[D. Baraglia and  L.P. Schaposnik]{Real structures on moduli spaces of Higgs bundles}
\author{David Baraglia and Laura P. Schaposnik }

\address{School of Mathematical Sciences, The University of Adelaide, SA 5005 Australia.}
\email{david.baraglia@adelaide.edu.au}
\address{Department of Mathematics, University of Illinois, Urbana, IL 61801, USA}
\email{schapos@illinois.edu}

\begin{document}

\baselineskip=1.1\baselineskip

\begin{abstract}
We construct triples of commuting real structures on the moduli space of Higgs bundles, whose fixed loci are branes of type $(B,A,A)$, $(A,B,A)$ and $(A,A,B)$. We study the real points through the associated spectral data and describe the topological invariants involved using $KO$, $KR$ and equivariant $K$-theory.
\end{abstract}

\maketitle

%%%%%%%%%%%%%%%%%%%%%%%%%%%%%%%%%%%%%%%%%%%%%%%%%%%%%%%%%%%%%%%%%%

\section{Introduction}\label{sec:intro}

The moduli space $\mathcal{M}_G(\Sigma)$ of $G$-Higgs bundles on a compact Riemann surface $\Sigma$ is the space of solutions to the gauge theoretic Hitchin equations on the  surface, where $G$ is a complex reductive Lie group. The smooth locus of $\mathcal{M}_G(\Sigma)$ is a hyperk\"ahler manifold, so there are complex structures $I,J,K$ obeying the same relations as the imaginary quaternions. This paper is concerned with the study of several naturally defined real structures on this moduli space. In fact, as it is impossible for an involution to be anti-holomorphic in all three complex structures it is natural to consider not one but three real structures simultaneously. We introduce a naturally defined triple of commuting real structures $i_1,i_2,i_3$ on $\mathcal{M}_G(\Sigma)$. We give geometric interpretations for these real structures and use spectral data to build up a detailed picture of their fixed point sets. Along the way we encounter various forms of $K$-theory as a convenient tool for studying the connected components of these fixed point sets.\\

The real structures $i_1,i_2,i_3$ are defined in Section \ref{sec:real}. The involution $i_1$ is defined by taking a real form $G^\sigma$ of $G$. Amongst the fixed points of $i_1$ are solutions to the Hitchin equations with holonomy in the real form $G^\sigma$. Moreover, when $G^\sigma$ is the split real form, we prove in Theorem \ref{proporder2} that the fixed points of $i_1$ are points of order $2$ in the fibres of  the Hitchin fibration, as seen in \cite{thesis}. The involution $i_2$, introduced in \cite{BS13} is defined by choosing a real structure on $\Sigma$, an anti-holomorphic involution $f : \Sigma \to \Sigma$. Amongst the fixed points of $i_2$ are representations of the orbifold fundamental group of the action of $f$ on $\Sigma$. Combining these two involutions we obtain a third involution $i_3 = i_1 \circ i_2$. Amongst the fixed points of $i_3$ are pseudo-real Higgs bundles, introduced in \cite{biswas2}. The involutions $i_2,i_3$ are well-behaved with respect to the Hitchin fibration and by restriction we find that their fixed point sets are real integrable systems. We also give a description of the fixed point sets in terms of orbifold representations in Section \ref{secrep}.\\

Section \ref{sec:gl} considers in detail the case of the general linear group $G = GL(m,\mathbb{C})$, with real form $GL(m,\mathbb{R}) \subset GL(m,\mathbb{C})$. In this case a Higgs bundle $(V,\Phi)$ is a rank $m$ holomorphic vector bundle $V$ and a holomorphic $(1,0)$-form valued endomorphism $\Phi$ of $V$. If ${\rm deg}(V)=0$, then as we recall in the paper polystable Higgs bundles $(V,\Phi)$ correspond to bundles with flat connection $(V,\nabla)$, where $\nabla$ has reductive holonomy. We say that $\nabla$ is simple if the only constant endomorphisms of $V$ are multiples of the identity. Restricting to simple, reductive holonomy we find:

\begin{itemize}
\item{Fixed points of $i_1$ are flat bundles with holonomy in $GL(m,\mathbb{R})$ or $GL(m/2,\mathbb{H})$.}
\item{Fixed points of $i_2$ are flat bundles with involution $\varphi : V \to V$ covering $f$ and preserving $\nabla$.}
\item{Fixed points of $i_3$ are flat bundles with anti-linear isomorphism $\varphi : V \to V$, covering $f$, preserving $\nabla$ and with $\varphi^2 = \pm 1$.}
\end{itemize}

As a first step towards identifying the connected components of the fixed point sets, we may consider the topological data associated to the underlying bundle as follows:
\begin{itemize}
\item{For $i_1$, the bundle $V$ carries a real or quaternionic structure, thus defines a class $[V]$ in $KO^0(\Sigma)$ or $KSp(\Sigma)$, real or quaternionic $K$-theory.}
\item{For $i_2$, the bundle $V$ carries a lift of the $\mathbb{Z}_2$-action on $\Sigma$, thus defines a class $[V] \in K^0_{\mathbb{Z}_2}(\Sigma)$ in $\mathbb{Z}_2$-equivariant $K$-theory.}
\item{For $i_3$, the bundle $V$ carries a real or quaternionic structure in the sense of Atiyah \cite{A66}, hence a class $[V]$ in $KR^0(\Sigma)$ or $KH^0(\Sigma)$.}
\end{itemize}

In Section \ref{sec:gl} we determine these $K$-theory groups and show that the $K$-theory classification is sharp in the sense that one can recover the bundle plus additional topological data up to isomorphism from the $K$-theory class.\\

As an example of the utility of these $K$-theory classes, we consider the following construction. Let $\overline{\Sigma} = \Sigma \times [-1,1]$ with involution $\tau(x,t) = (f(x),-t)$. The quotient $M = \overline{\Sigma}/\tau$ is a $3$-manifold with boundary $\partial M = \Sigma$. Then, we show the following:\\

\noindent \textbf{Theorem \ref{3man}} \textit{Let $(V,\nabla)$ be a fixed point of $i_3$ with simple holonomy. Then $\nabla$ extends over $M$ as a flat connection if and only if the class $[V] \in \tilde{K}^0_{\mathbb{Z}_2}(\Sigma)$ in reduced equivariant $K$-theory is trivial.}\\

In Section \ref{sec:spectral} we study the spectral data associated to fixed points. As we recall, spectral data consists of a branched cover $p : S \to \Sigma$ called the spectral curve and a line bundle $L \to S$. The bundle $V$ and Higgs field $\Phi$ are recovered by pushing down $L$ to $\Sigma$. For fixed points, the additional structure on $V$ determines similar structure on the spectral line $L$ and we recover the $K$-theory classes as push-forwards from the spectral curve to $\Sigma$. The push-forward in the $KO$-theory case was recently used by Hitchin for this purpose \cite{Char}. We complement this by determining explicit expressions for the push-forwards in equivariant $K$-theory and $KR$-theory in Theorems \ref{prop:eqpush} and \ref{prop:krpush}. The image of the push-forward maps $p_* : K_{\mathbb{Z}_2}^0(S) \to K_{\mathbb{Z}_2}^0(\Sigma)$ and $p_* : KR^0(S) \to KR^0(\Sigma)$ describe which topological classes of bundle with real structure can be given a reductive flat connection (with smooth spectral curve). This is the analogue of the Milnor-Wood type inequalities for the topological invariants obtained through $K$-theory.\\

Along the paper we adopt the physicists' language in which a Lagrangian submanifold is called an {\em A-brane} and a complex submanifold a {\em B-brane}. A submanifold of a hyperk\"ahler manifold may be of type $A$ or $B$ with respect to each of the complex structures and we may speak of branes of type $(B,B,B), (B,A,A), (A,B,A)$ and $(A,A,B)$. Under this classification the fixed point sets of $i_1,i_2,i_3$ are branes of type $(B,A,A),(A,B,A)$ and $(A,A,B)$ respectively. The main reason for considering branes is the connection to mirror symmetry and the geometric Langlands program. This program asserts that the moduli spaces $\mathcal{M}_G(\Sigma) , \mathcal{M}_{^L G}(\Sigma)$ are in duality, where $^L G$ is the Langlands dual group of $G$. According to this duality, specifically homological mirror symmetry, there should be an equivalence of categories of branes on $\mathcal{M}_G(\Sigma)$ and $\mathcal{M}_{^L G}(\Sigma)$ under which there are exchanges $(B,B,B) \leftrightarrow (B,A,A)$, $(A,B,A) \leftrightarrow (A,B,A)$ and $(A,A,B) \leftrightarrow (A,A,B)$. We conclude our work with Section \ref{sec:duality}, in which we speculate on how this duality acts on the fixed point sets of $i_1,i_2,i_3$.

%%%%%%%%%%%%%%%%%%%%%%%%%%%%%%%%%%%%%%%%%%%%%%%%%%%%%%%%%%%%%%%%%%

\vspace{0.1 in}

\noindent{\bf \ Acknowledgements:} The authors would like to thank S. Bradlow, O. Garc\'ia-Prada, N. Hitchin, and F. Schaffhauser for helpful comments.

\section{Higgs bundles}
We review Higgs bundles, the hyperk\"ahler structure of their moduli space and recall the Hitchin fibration. 

\subsection{The moduli space of $G$-Higgs bundles}

Let $\Sigma$ be a Riemann surface of genus $g > 1$ with canonical bundle $K$, and $G$ a complex Lie group with Lie algebra $\mathfrak{g}$. We shall assume throughout the paper that $G$ is reductive. Given a  principal $G$-bundle $P$ on $\Sigma$, we let $\mathfrak{g}_P$ denote the corresponding adjoint bundle. A $G$-Higgs bundle on a Riemann surface $\Sigma$ is a pair $(\overline{\partial}_A,\Phi)$, where $\overline{\partial}_A$ is a holomorphic structure on a principal $G$-bundle $P$ and $\Phi$ is a holomorphic section of $\mathfrak{g}_P \otimes K$. In the case of $G = GL(m,\mathbb{C})$ a $G$-Higgs bundle is equivalent to a classical Higgs bundle $(V,\Phi)$, consisting of a rank $m$ holomorphic bundle $V$ and a holomorphic map $\Phi : V \to V \otimes K$.\\

In order to define a moduli space of such pairs we shall recall the notions of stability and $S$-equivalence. Let $\mu(V)= \deg(V)/{\rm rk}(V)$ be the slope of the vector bundle $V$. We say that a $GL(m,\mathbb{C})$-Higgs bundle $(V,\Phi)$ is semi-stable if for every subbundle $W\subset V$ such that $\Phi(W)\subset W\otimes K$ we have $\mu(W)\leq \mu(V)$, and it is stable if one has a strict inequality. If one can write $(V,\Phi)=(V_1,\Phi_1)\oplus\ldots \oplus (V_k,\Phi_k) $ for $(V_i,\Phi_i)$ stable pairs such that $\mu(V_i)= \mu(V)$, then we say the Higgs bundle is polystable. To define $S$-equivalence consider a strictly semi-stable Higgs bundle $(V, \Phi)$. As it is not stable, $V$ admits a subbundle $W\subset V$ of the same slope which is preserved by $\Phi$. If $W$ is a subbundle of $V$ of least rank and same slope which is preserved by $\Phi$, it follows that the pair $(W,\Phi)$ is stable.  Then, by induction one obtains a flag of subbundles
$W_{0}=0\subset W_{1}\subset \ldots \subset W_{r}=V$
where $\mu(W_{i}/W_{i-1})=\mu(V)$ for $1\leq i\leq r$, and where the induced Higgs bundles $(W_{i}/W_{i-1}, \Phi_{i})$ are stable. This is the \textit{Jordan-H\"{o}lder filtration} of $V$, and it is not unique.  However the graded object
\[{\rm Gr}(V,\Phi):=\bigoplus_{i=1}^{r}(W_{i}/W_{i-1},\Phi_{i})\]
is unique up to isomorphism. 
 Two semi-stable Higgs bundles $(V,\Phi)$ and $(V',\Phi')$ are said to be $S$-equivalent if ${\rm Gr}(V,\Phi)\cong {\rm Gr}(V',\Phi')$. For a stable pair $(V,\Phi)$ the associated graded object coincides with $(V,\Phi)$ and the $S$-equivalence class is just the isomorphism class of the original pair. More generally each $S$-equivalence class contains a unique polystable object.
 
 Through the above definitions, one may construct the moduli space $\mathcal{M}_m^d$ of $S$-equivalence classes of classical semi-stable Higgs bundles of rank $m$ and degree $d$, or equivalently, the moduli space of polystable rank $m$ degree $d$ Higgs bundles. This space is a quasi-projective scheme of complex dimension $2m^2(g-1) + 2$ and contains an open subscheme ${\mathcal{M}'}_m^d$ corresponding to the moduli scheme of stable pairs \cite{nit}. When $m$ and $d$ are coprime semi-stable implies stable and the moduli space $\mathcal{M}_m^d$ is smooth.\\

By extending the stability notions to $G$-Higgs bundles, one can define stable, semi-stable and polystable $G$-Higgs bundles (see e.g., \cite{biswas}, \cite{bgm}). Then one can construct a corresponding moduli space of polystable $G$-Higgs bundles, denoted $\mathcal{M}_{G}$ or $\mathcal{M}_G(\Sigma)$. The dimension of $\mathcal{M}_G$ is $2 {\rm dim}(G) (g-1)$. The connected components of $\mathcal{M}_G$ are in bijection with isomorphism classes of principal $G$-bundles and these are parametrised by $\pi_1(G)$ \cite{dopa}. For $d \in \pi_1(G)$ we let $\mathcal{M}_G^d$ denote the corresponding connected component of $\mathcal{M}_G$.

%%%%%%%%%%%%%%%%%%%%%%%%%%%%%%%%%%%%%%%%%%%%%%%%%%%%%%%%

\subsection{Hyperk\"ahler structure on $\mathcal{M}_G$}\label{sec:hyper}

We shall briefly recall here the construction of a hyperk\"ahler metric on $\mathcal{M}_G$, obtained by an infinite dimensional hyperk\"ahler reduction. For simplicity we consider the case where $G$ is semi-simple. The reductive case requires only minor alterations such as modifying the Hitchin equations to allow for projectively flat connections.

Fix a hermitian metric $g$ on $\Sigma$ and an anti-holomorphic involution $\rho : G \to G$ whose fixed point set $G^\rho$ gives the compact real form of $G$. Let $P$ be a principal $G$-bundle and fix a reduction of structure to $G^\rho$. The reduction of structure determines a corresponding anti-linear involution $\rho : \mathfrak{g}_P \to \mathfrak{g}_P$ on the adjoint bundle. Given a section $x$ of  $\mathfrak{g}_P$ we write $x^*$ for $-\rho(x)$. We shall denote by $k( \, , \, )$  the Killing form on $\mathfrak{g}$.\\

The space $\mathcal{A}$ of holomorphic structures on $P$ is an affine space over $\Omega^{0,1}(\Sigma, \mathfrak{g}_P)$, hence the cotangent bundle $T^*\mathcal{A}$ is  an infinite dimensional flat hyperk\"ahler manifold. The tangent space to $T^*\mathcal{A}$ at any point can be  naturally identified with the direct sum  $\Omega^{0,1}(\Sigma , \mathfrak{g}_P  ) \oplus \Omega^{1,0}(\Sigma , \mathfrak{g}_P )$ and we shall denote by $(\Psi_i, \Phi_i)$ tangent vectors to this space. In terms of this identification the metric on $T^*\mathcal{A}$ is given by
\begin{equation*}
g( (\Psi_1,\Phi_1) , (\Psi_1 , \Phi_1) ) = 2i \int_{\Sigma} k( \Psi_1^* , \Psi_1) - k( \Phi_1^* , \Phi_1).
\end{equation*}
Moreover there are compatible complex structures $I,J,K$ defined by
\begin{equation*}
\begin{aligned}
 I (\Psi_1,\Phi_1) &= (i\Psi_1,i\Phi_1), \\
J (\Psi_1,\Phi_1) &= (i\Phi_1^*,-i\Psi_1^*), \\
K (\Psi_1,\Phi_1) &= (-\Phi_1^*,\Psi_1^*),
\end{aligned}
\end{equation*}
satisfying the usual quaternionic relations. This defines the hyperk\"ahler structure on $T^*\mathcal{A}$. We shall denote by $\omega_I,\omega_J,\omega_K$ the corresponding K\"ahler forms
\begin{equation*}
\omega_I(X,Y) := g(IX,Y)~,~\omega_J(X,Y) := g(JX,Y)~,~ \omega_K(X,Y) := g(KX,Y). 
\end{equation*}

For a pair $(\overline{\partial}_A , \Phi) \in T^*\mathcal{A}$ let $\nabla_A = \partial_A + \overline{\partial}_A = \rho(\overline{\partial}_A) + \overline{\partial}_A$ be the associated unitary connection and $F_A$ the curvature of $\nabla_A$. The group $\mathcal{G}^\rho$ of unitary gauge transformations acts on $T^*\mathcal{A}$ preserving the hyperk\"ahler structure. This action has a hyperk\"ahler moment map $\mu( \overline{\partial}_A , \Phi ) = ( F_A + [\Phi , \Phi^*] , \overline{\partial}_A \Phi )$. The hyperk\"ahler quotient $\mu^{-1}(0)/ \mathcal{G}^\rho$ of this action is then the moduli space of solutions to the Hitchin equations
\begin{equation}\label{Heq}
F_A + [\Phi , \Phi^*] = 0, \; \; \; \;  \overline{\partial}_A \Phi = 0,
\end{equation}
modulo unitary gauge transformations. From this we obtain a hyperk\"ahler structure on the smooth points of the moduli space of solutions to the Hitchin equations.\\

For $G$ semi-simple it is a result of Hitchin \cite{N1} and Simpson \cite{simpson88} that a $G$-Higgs bundle $(\overline{\partial}_A,\Phi)$ is gauge equivalent to a solution of the Hitchin equations (\ref{Heq}) if and only if it is polystable. This is used to establish an isomorphism between the moduli space $\mathcal{M}_G$ of polystable $G$-Higgs bundle and the moduli space of solutions to the Hitchin equations. In particular this gives a hyperk\"ahler structure on the smooth points of $\mathcal{M}_G$.\\

A solution to the above Hitchin equations (\ref{Heq}) defines an associated flat $G$-connection $\nabla = \nabla_A + \Phi + \Phi^*$. From the results of Donaldson \cite{donald} and Corlette \cite{cor}, the mapping $(\overline{\partial}_A , \Phi) \mapsto \nabla$ gives an isomorphism between the Higgs bundle moduli space $\mathcal{M}_{G}$ and the character variety ${\rm Hom}^+( \pi_1(\Sigma) , G)/G$ of reductive representations of $\pi_1(\Sigma)$ in $G$ (the definition of reductive representations and further details are recalled in Section \ref{secrep}). We say that a polystable Higgs bundle $(\overline{\partial}_A , \Phi)$ is simple if the only covariant constant gauge transformations of the associated connection $\nabla$ are those valued in the centre $Z(G)$ of $G$.  In particular, one has that simple Higgs bundles give smooth points on the moduli space \cite{ric}.\\

The hyperk\"ahler quotient construction carries over to the case of a reductive group requiring only a small modification. Consider for example the case $G = GL(m,\mathbb{C})$. The Hitchin equations for a Higgs bundle pair $(V,\Phi)$ should be modified to 
\begin{equation*}
F_A + [\Phi , \Phi^*] = -2 \pi i \mu(V) vol_\Sigma, \; \; \; \;  \overline{\partial}_A \Phi = 0,
\end{equation*}
where $\mu(V)$ is the slope of $V$ and $vol_\Sigma$ the volume form on $\Sigma$. The associated connection $\nabla = \nabla_A + \Phi + \Phi^*$ is now only projectively flat and we obtain a representation of a central extension of $\pi_1(\Sigma)$.

%%%%%%%%%%%%%%%%%%%%%%%%%%%%%%%%%%%%%%%%%%%%%%%%%%

\subsection{The Hitchin fibration}

The moduli space $\CM_{G}$ has a natural complex Lagrangian fibration over a vector space $\mathcal{A}_G$. To define this fibration let $p_1, \dots , p_k$ be a homogeneous basis for the algebra of invariant polynomials on $\mathfrak{g}$, of degrees $d_1, \dots , d_k$. Following \cite{N2}, the Hitchin fibration is given by
\begin{eqnarray*}
 h~:~ \mathcal{M}_{G}&\longrightarrow&\mathcal{A}_{G}:=\bigoplus_{i=1}^{k}H^{0}(\Sigma,K^{d_{i}}),\\
 (\overline{\partial}_A,\Phi)&\mapsto& (p_{1}(\Phi), \ldots, p_{k}(\Phi)).
\end{eqnarray*}
The map $h$, referred to as the Hitchin map, is a proper map for any choice of basis (see \cite[Section 4]{N2} for details). Given $d \in \pi_1(G)$ consider the restricted Hitchin map $h : \mathcal{M}_G^d \to \mathcal{A}_G$. For each component $\mathcal{M}_G^d$ the smooth fibres of $h$ are connected \cite{dopa} complex Lagrangian submanifolds with respect to the holomorphic symplectic form $\Omega_I = \omega_J + i\omega_K$. We have ${\rm dim}( \mathcal{A}_G) ={\rm dim}(  \mathcal{M}_{G})/2$, and the Hitchin map gives each component $\mathcal{M}_G^d$ the structure of an algebraically completely integrable system \cite{N2}. In particular $h$ is generically a submersion and the generic fibres are abelian varieties.

%%%%%%%%%%%%%%%%%%%%%%%%%%%%%%%%%%%%%%%%%%%%%%%%%%%%%%%%%%%%%%%%%%

\section{Real structures}\label{sec:real}

Having defined the moduli space $\mathcal{M}_{G}(\Sigma)$ of $G$-Higgs bundles on a compact Riemann surface $\Sigma$, we shall now consider the three natural involutions $i_1,i_2$ and $i_3$ on it and study their fixed point sets. 

\subsection{The three involutions}

The moduli space $\mathcal{M}_{G}(\Sigma)$ admits several distinct real structures, a phenomenon related to its hyperk\"ahler geometry. First consider a real form of $G$, given by the fixed point set $G^\sigma$ of an anti-holomorphic involution $\sigma : G \to G$. For any real form $\sigma$ we can find an anti-holomorphic involution $\rho : G \to G$ commuting with $\sigma$, whose fixed point set defines the compact real form of $G$ (for details, see \cite{helga}). The Cartan involution of the real form $G^\sigma$ is the holomorphic involution $\theta = \rho \circ \sigma$. From $\sigma$ we obtain an involution $i_1$ on the Moduli space of Higgs bundles, given by
\begin{equation*}
i_1(\bar \partial_A, \Phi)=(\theta(\bar \partial_A),-\theta( \Phi)).
\end{equation*}
The involution $i_1$ is an isometry of the hyperk\"ahler metric, holomorphic in $I$ and anti-holomorphic in $J$, $K$. Therefore its fixed point set is a brane of type $(B,A,A)$.\\
 
 A second way to obtain a real structure on $\mathcal{M}_{G}(\Sigma)$ is to consider real structures on the surface $\Sigma$. Let $f : \Sigma \to \Sigma$ be an anti-holomorphic involution, a real structure on $\Sigma$. Viewing $\mathcal{M}_{G}(\Sigma)$ as a moduli space of connections, the action of $f$ on $\Sigma$ induces an involution $i_2$ by pullback of connections. In terms of Higgs bundles, $i_2$ is given by
\begin{equation*}
i_2(\bar \partial_A, \Phi)=(f^*(  \partial_A),f^*( \Phi^*  ))= (f^*(\rho(\bar \partial_A)), -f^*( \rho(\Phi) )).
\end{equation*}
We have seen in \cite{BS13} that $i_2$ is an isometry which is holomorphic in $J$ and anti-holomorphic in $I$, $K$. Therefore its fixed point set is a brane of type $(A,B,A)$.\\

Lastly, by combining a real structure on the group $G$ with a real structure on the surface $\Sigma$, we obtain a third class of involution $i_3$, given by
\begin{equation*}
i_3(\bar \partial_A, \Phi)=(f^* \sigma(\bar \partial_A),f^*\sigma( \Phi)).
\end{equation*}
Since $i_3 = i_1 \circ i_2$ we have that $i_3$ is an isometry, holomorphic in $K$ and anti-holomorphic in $I$, $J$. Therefore its fixed point set is a brane of type $(A,A,B)$.

%%%%%%%%%%%%%%%%%%%%%%%%%%%%%%%%%%%%%%%%%%%%%%%%%%%%%%%%%%

\subsection{Fixed point sets}

The fixed point sets of $i_1,i_2,i_3$ meet the smooth points of $\mathcal{M}_G$ in complex Lagrangian submanifolds. The case of $i_2$ was established in \cite{BS13} and the same argument applies to the other involutions. The fixed point sets have interpretations in terms of the corresponding holonomy representations. The simplest case are fixed points of $i_2$, which in terms of connections are simply those connections $\nabla$ which are isomorphic to their pullback $f^*\nabla$. For the involution $i_3$ we have fixed points given by pseudo-real Higgs bundles as defined in \cite{biswas2}.

To discuss fixed points of $i_1$, recall that the involutions $\sigma,\rho,\theta$ induce corresponding involutions on the Lie algebra $\mathfrak{g}$. We shall denote by  $\mathfrak{g}^\sigma$ the fixed point set of $\sigma$, the Lie algebra of $G^\sigma$. We obtain a decomposition $\mathfrak{g}^\sigma = \mathfrak{u} \oplus \mathfrak{m}$ of $\mathfrak{g}^\sigma$ into the $\pm 1$-eigenspaces of $\theta|_{\mathfrak{g}^\sigma}$, the Cartan decomposition of $\mathfrak{g}^\sigma$. Let $\mathfrak{u}^c = \mathfrak{u} \oplus i \mathfrak{u}$ and $\mathfrak{m}^c = \mathfrak{m} \oplus i \mathfrak{m}$ be the complexifications of $\mathfrak{u}$ and $\mathfrak{m}$. Let $U$ be the maximal compact subgroup of $G^\sigma$ and $U^c$ its complexification. Higgs bundles with holonomy in the real form $G^\sigma$ are pairs $(\overline{\partial}_A , \Phi)$, where $\overline{\partial}_A$ is a holomorphic structure on a principal $U^c$-bundle $P'$ and $\Phi$ is a holomorphic section of $(P' \times_{U^c} \mathfrak{m}) \otimes K$. Clearly such Higgs bundles are fixed points of $i_1$. In general, there are fixed points of $i_1$ and $i_3$ other than those we have just described. We address the problem of classifying all fixed points in Section \ref{secrep} in terms of special classes of representations.\\

We consider the action of the real structures on the connected components of $\mathcal{M}_G$. For this note that the Cartan involution $\theta : G \to G$ induces an involution $\theta_* : \pi_1(G) \to \pi_1(G)$.
\begin{proposition}
The action of the involutions  $i_1,i_2,i_3$ on the space of connected components  $\pi_0(\mathcal{M}_G) \simeq \pi_1(G)$ is given by
\begin{equation*}
i_1(d) = \theta_*(d), \; \; \; \; i_2(d) = -d, \; \; \; \; i_3(d) = -\theta_*(d),
\end{equation*}
where $d \in \pi_1(G)$.
\end{proposition}
\begin{proof}
Let $P$ be a principal $G$-bundle on $\Sigma$. We define $\theta(P)$ to be the principal bundle diffeomorphic to $P$ but with $G$-action $(p , g) \mapsto p~  \theta(g)$. At the level of isomorphism classes this gives the map $\theta_* : \pi_1(G) \to \pi_1(G)$. From the definition of $i_1$ it is clear that it sends a principal $G$-bundle $P$ to $\theta(P)$. It was shown in \cite{BS13} that the action of $i_2$ on principal $G$-bundles is $d \mapsto -d$. The result for $i_3$ follows since $i_3 = i_1 \circ i_2$.
\end{proof}

When $\sigma = \rho$ is the compact real form, we find a close relationship between the fixed point sets of $i_2$ and $i_3$.
\begin{proposition}\label{propmultbyi}
Let $\sigma = \rho$ be the compact real form. Then the fixed point sets of $i_2,i_3$ are diffeomorphic.
\end{proposition}
\begin{proof}
Let $i : \mathcal{M}_{G} \to \mathcal{M}_{G}$ denote the action of $i \in \mathbb{C}^*$ sending $(\overline{\partial}_A , \Phi )$ to $(\overline{\partial}_A , i\Phi)$. Then if $\sigma = \rho$ we see that $i_3 = i \circ i_2 \circ i^{-1}$. It follows that $i$ exchanges the fixed point sets of the involutions $i_2,i_3$.
\end{proof}

By restriction $J$ defines a complex structure on the fixed point set of $i_2$. Similarly $K$ defines a complex structure on the fixed point set of $i_3$. Moreover we have $i \circ J = K \circ i$, so the isomorphism in Proposition \ref{propmultbyi} is in fact a complex analytic isomorphism. In Section \ref{secrep}, we see that the fixed point sets for $i_2,i_3$ correspond to quite different classes of representations. It is then unexpected to find a natural bijection between these spaces.\\

When $\sigma$ is the split real form the involution $i_1$ admits a very simple description in terms of the Hitchin fibration. For $d \in \pi_1(G)$, we have that $i_1$ sends $\mathcal{M}_G^d$ to $\mathcal{M}_G^{\theta_*(d)}$. Thus fixed points can only occur when $\theta_*(d) = d$. Suppose now that $\theta_*(d) = d$ and let $F = h^{-1}(a)$ be a non-singular fibre of $h : \mathcal{M}_G^d \to \mathcal{A}_G$ corresponding to a point $a \in \mathcal{A}_G$. From \cite[Chapter 4]{thesis} we have the following: 

\begin{theorem}\label{proporder2}
The action of $i_1$ preserves the Hitchin fibration, $h \circ i_1 = h$. Let $d \in \pi_1(G)$ with $\theta_*(d) = d$ and let $F$ be a non-singular fibre of $\mathcal{M}_G^d$. The restriction $ i_1|_F : F \to F$ of $i_1$ to $F$ has fixed points. Let $m \in F$ be a fixed point. Then the action of $i_1|_F$ on $F$ is given by $x \mapsto -x$ with respect to the origin defined by $m$.
\end{theorem}
\begin{proof}
For the split real form the map $\Phi \mapsto -\theta(\Phi)$ preserves the invariant polynomials, hence $h \circ i_1 = h$. Then for any $d \in \pi_1(G)$ with $\theta_*(d) = d$, we have that $i_1$ acts on the fibres of $\mathcal{M}_G^d \to \mathcal{A}_G$. Let $F = h^{-1}(a)$ be a non-singular fibre of $h$ over a point $a \in \mathcal{A}_G$. As the fibres are connected \cite{dopa} we have that $F$ is a complex torus and the restriction $i_1|_F : F \to F$ is a complex automorphism, since $i_1$ is holomorphic in $I$.\\

Recall that the Hitchin fibration is Lagrangian with respect to $\Omega_I = \omega_J + i\omega_K$. Let $\alpha^1 , \dots , \alpha^k$ be a complex basis for $T_a^*\mathcal{A}_G$ and define vector fields $X_1 , \dots , X_k$ on $F$ by requiring $i_{X_i} \Omega_I = h^* \alpha^i $. These vector fields are commuting and integrate to an action of $\mathbb{C}^k$ on $F$ by translation. The stabilizer of this action is a lattice $\Lambda \subseteq \mathbb{C}^k$ giving an identification $F = \mathbb{C}^k / \Lambda$. Using the fact that $i_1$ is an isometry which is anti-holomorphic in $J,K$ we have $i_1^* \Omega_I = -\Omega_I$. On the other hand since $h \circ i_1 = h$ we have $i_1^* (h^* \alpha^i)  = h^*\alpha^i$. It follows that ${i_1}_*(X_i) = -X_i$ for each $i$. This is enough to ensure that $i_1$ has a fixed point $m \in F$. Using $m$ as an origin, we have by exponentiation that the action of $i_1$ on $F$ is given by $x \mapsto -x$.
\end{proof}

\begin{remark} In particular, for each $d \in \pi_1(G)$ with $\theta_*(d) = d$, the intersection of the fixed points of $i_1$ on $\mathcal{M}_G^d$ with the smooth fibres of the Hitchin fibration $h: \mathcal{M}_{G}^d\rightarrow \mathcal{A}_{G}$ is given by the elements of order two in those fibres.
\end{remark}

The involutions $i_2,i_3$ are well-behaved with respect to the Hitchin fibration $\mathcal{M}_G \to \mathcal{A}_G$ and give rise to real integrable systems. In \cite{BS13} we saw that there exists an anti-linear involution $f_2 : \mathcal{A}_G \to \mathcal{A}_G$ such that $h \circ i_2 = f_2 \circ h$. The details of this proof carry over without difficulty to the case $i_3$ and we obtain a second anti-linear involution $f_3 : \mathcal{A}_G \to \mathcal{A}_G$ such that $h \circ i_3 = f_3 \circ h$. Let $\mathcal{L}_{G} \subset \mathcal{M}_G$ be the fixed point set of $i_2$ and $L_G \subset \mathcal{A}_G$ the fixed point set of $f_2$. The restriction of $\omega_J$ to $\mathcal{L}_G$ gives a symplectic structure and $h : \mathcal{L}_G \to L_G$ is a Lagrangian fibration \cite{BS13}. The non-singular fibres are a disjoint union of tori, though the number of components of the fibres generally varies as one moves around the base $L_G$. The proof applies just as well to the case $i_3$. Let $\mathcal{L}'_G$ be the fixed point set of $i_3$ and $L'_G$ the fixed point set of $f_3$. The restriction of $\omega_K$ to $\mathcal{L}'_G$ gives a symplectic structure and $h : \mathcal{L}'_G \to L'_G$ is a Lagrangian fibration. Again, the non-singular fibres are unions of tori.

The situation for $i_1$ is different. In this case the action of $i_1$ on $\mathcal{M}_G$ covers the involution $f_1 = f_2 \circ f_3 : \mathcal{A}_G \to \mathcal{A}_G$. Note that $f_1$ is linear and thus its fixed point subspace need not be half dimensional. Moreover it is possible that the fixed point subspace of $f_1$ could lie entirely within the singular locus of the Hitchin fibration, so that the restriction of $h$ to the fixed point set of $i_1$ may be poorly behaved (for example this is the case when $G=SL(2m, \mathbb{C}), SO(4m, \mathbb{C})$, $Sp(4m, \mathbb{C})$ and $G^{\sigma} = SL(m, \mathbb{H}), SO(2m, \mathbb{H})$, $ Sp(m, m)$ respectively, studied in \cite{hitchin13}).

\begin{remark}
In the case of the involution $i_3$, the set of fixed points in the moduli space of classical Higgs bundles  for which $\Phi=0$ was thoroughly studied in \cite{florent}.
\end{remark}

%%%%%%%%%%%%%%%%%%%%%%%%%%%%%%%%%%%%%%%%%%%%%%%%%%%%%%%%%%%%%%%%%%%%%%%

\section{General linear case and $K$-theory}\label{sec:gl}

In this section we consider the case where $G = GL(m,\mathbb{C})$ with real structure $\sigma(A) = \overline{A}$ and Cartan involution $\theta(A) = (A^t)^{-1}$. Here a Higgs bundle $(V,\Phi)$ is a rank $m$ holomorphic vector bundle $V$ and holomorphic map $\Phi : V \to V \otimes K$. We will see that $K$-theory in various forms helps to distinguish between connected components of the fixed point sets of the real structures $i_1,i_2,i_3$.

%%%%%%%%%%%%%%%%%%%%%%%%%%%%%%%%%%%%%%

\subsection{$KO$-theory and $i_1$}

Fixed points of $i_1$ may be obtained from solutions to the Hitchin equations with holonomy in the corresponding real form $G^\sigma = GL(m,\mathbb{R})$. For Higgs bundles with holonomy in $GL(m,\mathbb{R})$ we may take $V$ to be a holomorphic $O(m,\mathbb{C})$-bundle with reduction to the maximal compact $O(m)$ and $\Phi$ to be symmetric. Such a Higgs bundle $(V,\Phi)$ is a fixed point of $i_1$ since the orthogonal structure gives an isomorphism between $(V,\Phi)$ and $i_1(V,\Phi) = (V^* , \Phi^t)$. Since the vector bundle $V$ has a real structure it defines a class $[V] \in KO^0(\Sigma)$ in the $KO$-theory of $\Sigma$. We have an isomorphism
\begin{equation*}
KO^0(\Sigma) = \mathbb{Z} \oplus H^1(\Sigma , \mathbb{Z}_2) \oplus H^2(\Sigma , \mathbb{Z}_2)
\end{equation*}
given by $[V] \mapsto ( {\rm rank}(V) , w_1(V) , w_2(V) )$, so the relevant topological invariants are the Stiefel-Whitney classes $w_1(V) , w_2(V)$ of $V$.\\

When $m = 2m'$ is even a second class of fixed points of $i_1$ are given by Higgs bundles with holonomy in $GL(m',\mathbb{H})$. For this $V$ is a holomorphic bundle with $Sp(2m' , \mathbb{C})$-structure and $\Phi = \Phi^t$ using the symplectic transpose \cite{hitchin13}. The symplectic structure then gives an isomorphism $(V,\Phi) \simeq (V^*,\Phi^t)$ so that such Higgs bundles are fixed points of $i_1$. The symplectic structure on $V$ defines a class $[V] \in KSp(\Sigma)$, the Grothendieck group of bundles with symplectic structure on $\Sigma$. However since $Sp(m')$ is simply connected all symplectic bundles are trivial and $KSp(\Sigma) = \mathbb{Z}$, the only invariant being the rank of the bundle. We will see in Section \ref{secrep} that a fixed point of $i_1$ with simple holonomy must belong to one of these two classes.

%%%%%%%%%%%%%%%%%%%%%%%%%%%%%%%%%%%%%%%%%%%%%%%%%%%%%

\subsection{Equivariant $K$-theory and $i_2$}

Let $(V,\Phi)$ be a fixed point of $i_2$ with simple holonomy. This requires $V$ to have degree zero, so the associated connection $\nabla = \nabla_A + \Phi + \Phi^*$ is flat. Fixed points of $i_2$ correspond to connections $\nabla$ such that $f^* \nabla \simeq \nabla$, for $f:\Sigma\rightarrow \Sigma$ an anti-holomorphic involution as introduced in Section \ref{sec:real}. Equivalently there is a bundle isomorphism $\varphi : V \to V$ covering $f$ which respects $\nabla$. Thus $\varphi^2$ is covariantly constant, so assuming $\nabla$ is simple we have that $\varphi^2 = c$ for some constant $c \in \mathbb{C}^*$. Choosing a square root $c^{1/2}$ and replacing $\varphi$ by $\varphi c^{-1/2}$ we obtain an involution $\varphi : V \to V$ lifting $f$ and preserving $\nabla$. By simplicity the lift $\varphi$ is unique up to an overall sign. The $\mathbb{Z}_2$-action of $\varphi$ gives $V$ the structure of a $\mathbb{Z}_2$-equivariant vector bundle and thus defines a class $[V] \in K^0_{\mathbb{Z}_2}(\Sigma)$ in equivariant $K$-theory. The class $[V]$ is determined from $(V,\Phi)$ up to a sign ambiguity which corresponds to replacing $\varphi$ by $-\varphi$.\\

If $f$ acts without fixed points, the quotient $\Sigma' = \Sigma / f$ is a compact non-orientable surface. Equivariant bundles are obtained by pullback from $\Sigma'$ and we have
\begin{equation*}
\widetilde{K}^0_{\mathbb{Z}_2}(\Sigma) = \widetilde{K}^0(\Sigma') = \mathbb{Z}_2.
\end{equation*}
This defines a $\mathbb{Z}_2$-invariant with the property that it is non-zero for the unique non-trivial rank $m$ bundle on $\Sigma'$. When $f$ has $n > 0$ fixed components we have, by the Mayer-Vietoris sequence 
\begin{equation}\label{equkz2fix}
\widetilde{K}^0_{\mathbb{Z}_2}(\Sigma) = \mathbb{Z}^n.
\end{equation}
Over each fixed component $S^1 \subset \Sigma$ of $f$ we may decompose $V|_{S^1}$ into the $\pm 1$-eigenspaces $V^{\pm}|_{S^1}$ of $\varphi$. Taking the dimension ${\rm dim}(V^{-}|_{S^1})$ gives a homomorphism $\widetilde{K}^0_{\mathbb{Z}_2}(\Sigma) \to \mathbb{Z}$ and taking the sum of these homomorphisms over the fixed components of $f$ gives the isomorphism (\ref{equkz2fix}). The group of $\mathbb{Z}_2$-equivariant line bundles is given by $H^2_{\mathbb{Z}_2}(\Sigma , \mathbb{Z}) = \mathbb{Z}_2^n$. For each assignment of $\pm 1$ to each fixed component of $f$, we can therefore find a corresponding equivariant line bundle unique up to isomorphism. This shows that exactly $(m+1)^n$ classes in $\widetilde{K}^0_{\mathbb{Z}_2}(\Sigma)$ are represented by rank $m$ equivariant bundles.\\

The equivariant $K$-theory classification is sharp in the following sense:
\begin{proposition}
Let $V,V'$ be equivariant vector bundles on $\Sigma$. Then $V,V'$ are isomorphic as equivariant vector bundles if and only if $[V] = [V'] \in K^0_{\mathbb{Z}_2}(\Sigma)$. 
\end{proposition}
\begin{proof}
In the case that $f$ has no fixed points this is trivial. Suppose now that $f$ has fixed points and let $V$ be an equivariant vector bundle. Let $\varphi : V \to V$ be the involution on $V$ and let $S^1_1, \dots , S^1_n \subset \Sigma$ be the fixed circle components of $f$. Suppose that for $i = 1, \dots , n$ we are given sections $s_i$ of $V|_{S^1_i}$ such that $\varphi(s_i) = \epsilon_i s_i$, where $\epsilon_i = \pm 1$. We claim that there exists a rank $1$ subbundle $L \subset V$ such that $\varphi(L) = L$ and the restriction of $L$ to $S^1_i$ is spanned by $s_i$. Let $L$ be the equivariant line bundle with eigenvalue $\epsilon_i$ over $S^1_i$. Considering $V \otimes L^*$ it suffices to assume $\epsilon_i = 1$ for all $i$. This is enough to ensure the existence of a section $s$ of $V$ which restricts to $s_i$ on $S^1_i$ and such that $\varphi(s_i) = s_i$. This proves our claim. Choosing a $\varphi$-invariant hermitian metric and taking the complement of $L$ in $V$ we have that $V$ is the sum of $L$ and a lower rank equivariant bundle. The proposition now follows by induction on rank.
\end{proof}

We give another description of equivariant $K$-theory, which though unconventional is well suited to spectral curves. Let $V$ be a complex vector bundle with anti-linear isomorphism $\psi : V \to V^*$ covering $f$. We say $\psi$ is symmetric if $(\psi(a))(b) = \overline{ (\psi(b))(a) }$ for all $a,b$. Let $h$ be a hermitian metric on $V$, taken anti-linear in the first factor. We can also view $h$ as an anti-linear isomorphism $h : V \to V^*$. We say $h$ is compatible with $\psi$ if $h(a,b) = \psi(a)( h^{-1} \psi(b) )$. Setting $\varphi = h^{-1} \circ \psi$, we have that $\varphi : V \to V$ is a linear involution covering $f$ and $(V,\varphi)$ becomes a class in $K^0_{\mathbb{Z}_2}(\Sigma)$. It is clear that this class does not depend on the choice of compatible hermitian metric $h$. Therefore we may identify $K^0_{\mathbb{Z}_2}(\Sigma)$ with the Grothendieck group of vector bundles $V$ equipped with a symmetric anti-linear isomorphism $\psi : V \to V^*$. Note that over the fixed points of $f$, the map $\psi$ defines an indefinite hermitian form $\langle a , b \rangle = (\psi(a))(b)$. The signature of this form over each component corresponds to the dimensions of the $\pm 1$-eigenspaces of the restriction of $\varphi$.\\

Consider the product $\overline{\Sigma}=\Sigma \times [-1,1]$ with involution $\tau(x,t) = (f(x),-t)$. The quotient  $M=\bar \Sigma /\tau$ is a 3-manifold with boundary $\partial M =\Sigma$. From $M$ we obtain a distinguished subspace of representations of $\pi_1(\Sigma)$, those representations which extend as flat connections from $\Sigma$ to $M$. Such representations are fixed points of $i_2$ \cite{BS13}. Moreover, we can use equivariant $K$-theory to characterise which fixed points of $i_2$ arise in this manner.
\begin{theorem}\label{3man}
Let $(V,\Phi)$ be a solution of the Hitchin equations with a lift of $f$ to an involution $\varphi : V \to V$ preserving the associated flat connection $\nabla$. Suppose also that $\nabla$ has simple holonomy. Then $\nabla$ extends over the 3-manifold $M$ if and only if the class $[V] \in \tilde{K}_{\mathbb{Z}_2}^0(\Sigma)$ is trivial (possibly on replacing $\varphi$ by $-\varphi$).
\end{theorem}
\begin{proof}
Let $\Sigma' = \Sigma/f$ be the quotient which is a surface with boundary. There is a homotopy equivalence $M \simeq \Sigma'$, so we can restate the problem in terms of finding a local system on $\Sigma'$ such that the pullback to $\Sigma$ gives the local system of constant sections of $\nabla$. If such a local system exists we must have that $\varphi$ or $-\varphi$ acts as the identity on each fixed component by simplicity. Conversely suppose that $\varphi$ acts as the identity on each fixed component. Then the local system of constant sections of $\nabla$ can be factored by $\varphi$ giving the desired local system on $\Sigma'$.
\end{proof}

%%%%%%%%%%%%%%%%%%%%%%%%%%%%%%%%%%%%%%%%%%%%5

\subsection{$KR$-theory and $i_3$}

Let $(V,\Phi)$ be a fixed point of $i_3$ with simple holonomy. Again it is easiest to describe fixed points in terms of the associated connection $\nabla$. In this case fixed points correspond to the condition that there is an anti-linear isomorphism $\varphi : V \to V$ covering $V$ and preserving $\nabla$. Since $\varphi^2$ is covariantly constant we have, assuming $\nabla$ is simple, that $\varphi^2 = c$ for some non-zero constant $c$. Then since $\varphi^3 = \varphi \circ c = c \circ \varphi$ we find $c$ is real. After rescaling we may assume that $c$ is either $1$ or $-1$. Since $\nabla$ is simple it follows that one and only one of these two cases can occur.\\

When $c = 1$ we have that $f$ lifts to an anti-linear involution $\varphi : V \to V$ preserving $\nabla$. This gives $V$ the structure of a real vector bundle in the sense of Atiyah \cite{A66} and a class $[V] \in KR^0(\Sigma)$. The map $\varphi$ is determined only up to rescaling by $U(1) \subset \mathbb{C}^*$, but this action is trivial on $KR^0(\Sigma)$ so that the class $[V]$ is independent of the choice of $\varphi$. Let $n$ be the number of components of the fixed point set of $f$. From the Mayer-Vietoris sequence we have an isomorphism \cite{kar}
\begin{equation*}
\widetilde{KR}^0(\Sigma) = \{ (d,x^1,\dots,x^n) \in \mathbb{Z} \oplus \mathbb{Z}_2^n \, | \, d = x^1 + \dots + x^n ({\rm mod} \, 2) \}.
\end{equation*}
For a bundle $V$ real with respect to $f$, let $d$ be the degree. For each fixed circle component $S^1 \subset \Sigma$ of $f$ the restriction $V|_{S^1}$ is a real line bundle over the circle which has a first Stiefel-Whitney class $w_1(V|_{S^1}) \in H^1(S^1,\mathbb{Z}_2) = \mathbb{Z}_2$. Let $x^1, \dots , x^n$ be the first Stiefel-Whitney classes of $V$ restricted to the fixed components. Then $[V]$ corresponds to $(d,x^1, \dots , x^n)$. Conversely every class in $\widetilde{KR}^0(\Sigma)$ can be represented by a real bundle of rank $m$ for any $m$.\\

When $c= -1$ we have that $f$ lifts to an anti-linear isomorphism $\varphi : V \to V$ preserving $\nabla$ and $\varphi^2 = -1$. This gives $V$ a quaternionic structure. Let $KH^0(\Sigma)$ denote the Grothendieck group of quaternionic vector bundles on $\Sigma$. Then, from \cite{dup} there is an isomorphism $KH^0(\Sigma) = KR^{-4}(\Sigma)$. Using this one finds 
\begin{equation*}
KH^0(\Sigma) = \mathbb{Z}^2.
\end{equation*}
When the anti-holomorphic involution $f$ on $\Sigma$ has no fixed points there exists a quaternionic line bundle $L$ of degree $g-1 ({\rm mod} \, 2)$. Tensoring by $L$ gives the isomorphism $KH^0(\Sigma) \simeq KR^0(\Sigma) = \mathbb{Z}^2$. When $f$ has fixed points any quaternionic bundle must have even rank $2r$ and even degree $2d$ \cite{bhh}. Moreover every such pair $(2r,2d)$ occurs for if $V$ is a complex bundle of rank $r$ and degree $d$, the bundle $V \oplus f^*(\overline{V})$ admits a quaternionic structure. In all cases we see that the $KH$-class of a quaternionic bundle is classified by rank and degree.\\

From \cite{bhh} it can be deduced that the $KR$-theory class of a real bundle $E$ on $\Sigma$ completely determines $V$ as a real bundle. The same is true for quaternionic bundles and $KH$-theory.

%%%%%%%%%%%%%%%%%%%%%%%%%%%%%%%%%%%%%%%%%%%%%%%%%%%%%%%%%%%%%%%%%%%%%%%

\section{Spectral data of the fixed points}\label{sec:spectral}

From \cite{N2}, principal $G$-Higgs bundles have associated some spectral data on which we can study the induced action of the involutions $i_1,i_2$ and $i_3$. In this section we study these fixed point sets in terms of the corresponding spectral data.

\subsection{Spectral curves}

The fibres of the Hitchin system are most easily seen using spectral curves, introduced in \cite{N2}. Given a classical Higgs bundle $(V,\Phi)$ of rank $m$, the spectral curve $S$ is the set of points $\lambda \in K$ satisfying the characteristic equation
\begin{equation*}
det( \lambda - \Phi) = \lambda^m + a_1 \lambda^{m-1} + \dots + a_m = 0,
\end{equation*}
where $a_i \in H^0(\Sigma , K^i)$. From Bertini's theorem $S$ is smooth for generic $a_1, \dots , a_m$. For a smooth spectral curve $S$ the projection $K \to \Sigma$ restricts to $S$ giving a branched $m$-fold cover $p : S \to \Sigma$ and the canonical bundle $K_S$ of $S$ is given by $p^*K^m$. The eigenspaces of $\Phi$ define a line bundle $U$ on $S$. More precisely, $U \otimes p^* K$ is the cokernel of $\eta - p^* \Phi$ in $p^*V\otimes p^*K$, where $\eta$ is the tautological section of $p^*K$ \cite{bobi}. The bundle $V$ is recovered as the direct image sheaf $V = p_*U$ and we recover $\Phi$ by pushing forward the endomorphism $\eta : U \to U \otimes p^* K$.

Conversely, given a generic point $a = (a_1, \dots , a_m ) \in \mathcal{A}_{GL(m,\mathbb{C})}$ and a line bundle $U$ over the corresponding spectral curve $S$, we obtain a Higgs bundle by the above construction. This identifies the fibre of the Hitchin map over $a$ with the Picard variety $Pic(S)$ of $S$. Let $K^{1/2}$ be a spin structure on $\Sigma$. It is convenient to define a new line bundle $L$ such that $U = L \otimes p^* K^{(m-1)/2}$. Then if $V = p_*(U)$, we have by Grothendieck-Riemann-Roch $\deg{V} = \deg{L}$. The map sending a line bundle $L \in Pic(S)$ to the vector bundle $V = p_*( L \otimes p^* K^{(m-1)/2})$ with corresponding Higgs field identifies the fibre of $\mathcal{M}_{GL(m,\mathbb{C})}^d$ over $a$ with $Pic_d(S)$, the space of degree $d$ line bundles on $S$. In particular the fibre of $\mathcal{M}_{GL(m,\mathbb{C})}^0$ may be identified with the Jacobian $Jac(S)$ of $S$.

%%%%%%%%%%%%%%%%%%%%%%%%%%%%%%%%%%%%%%%%%%%%%%%%%%%%%%%%%%%%

\subsection{Push-forward maps}

For this section we take $p : S \to \Sigma$ to be an $m$-sheeted branched cover, not necessarily given as a spectral curve and let $K_\Sigma, K_S$ denote the corresponding  canonical bundles. Let $K_\Sigma^{1/2},K_S^{1/2}$ be spin structures on $\Sigma$,$S$. Given a holomorphic vector bundle $W$ on $S$, set $U = W \otimes K_S^{1/2} \otimes p^* K_\Sigma^{-1/2}$ and $V = p_*(V)$. Note that $\deg{V} = \deg{W}$ and ${\rm rk}(V) = m . {\rm rk}(W)$.\\

For $KO$-theory we take $W$ to be a rank $k$ holomorphic bundle with complex orthogonal structure. Choosing a compatible hermitian structure on $W$ we have a reduction to $O(k)$, so $W$ defines a class $[W] \in KO^0(S)$. By relative duality we obtain an isomorphism $V \simeq V^*$, which is symmetric. Thus $V$ has an orthogonal structure and defines a class $[V] \in KO^0(\Sigma)$. This construction defines a push-forward homomorphism $p_* : KO^0(S) \to KO^0(\Sigma)$ in real $K$-theory. The choice of spin structures $K_\Sigma^{1/2}$, $K_S^{1/2}$ are the $KO$-orientations required to define the push-forward.\\

To consider the other $K$-theory groups suppose that $f$ lifts to an anti-holomorphic involution $\tilde{f} : S \to S$. We must choose spin structures compatibly. For this recall that to any compact Riemann surface $\Sigma$ with anti-holomorphic involution $f$, there exists a spin structure $K_\Sigma^{1/2}$ such that $f^*( K_\Sigma^{1/2}) \simeq \overline{K}_\Sigma^{1/2}$ \cite{A71}. In fact we have that $f$ lifts to an anti-linear map $\gamma_\Sigma : K_\Sigma^{1/2} \to K_\Sigma^{1/2}$ such that $\gamma_\Sigma \otimes \gamma_\Sigma : K_\Sigma \to K_\Sigma$ is the map $\omega \mapsto f^*(\overline{\omega})$. If $f$ has fixed points then $\gamma_\Sigma^2 =1$, and if $f$ has no fixed points then $\gamma^2_\Sigma = 1$ if $g$ is odd and $-1$ if $g$ is even \cite{A71}. In other words, $K_\Sigma$ admits a real or quaternionic square root $K_\Sigma^{1/2}$. Similarly choose such a spin structure $K_S^{1/2}$ on $S$ with anti-linear isomorphism $\gamma_S : K_S^{1/2} \to K_S$ covering $\tilde{f}$. If $S^1 \subset \Sigma$ is a fixed component of $f$ then $\gamma_\Sigma$ defines a real structure on the restriction of $K^{1/2}_\Sigma$ to $S^1$. By squaring we obtain a real non-vanishing section of $K_\Sigma$ along the circle and hence an orientation of the circle. Similarly we can use $\gamma_S$ to orient circles in $S$ fixed by $\tilde{f}$.\\

For equivariant $K$-theory we suppose that $W$ is a holomorphic vector bundle on $S$ with a symmetric anti-linear isomorphism $\psi : W \to W^*$ covering $\tilde{f}$. Recall that this defines a class $[W] \in K_{\mathbb{Z}_2}^0(S)$. For the push-forward construction we require $\psi$ to be holomorphic in the sense that if $\alpha$ is the germ of a holomorphic section of $W$ at $s \in S$, then $\psi \circ \alpha \circ \tilde{f}$ is the germ of a holomorphic section of $W^*$ as $\tilde{f}(s)$. Set $U = W \otimes K_S^{1/2} \otimes p^* K_\Sigma^{-1/2}$. Combining $\psi$ with $\gamma_S$, $\gamma_\Sigma$ we obtain an isomorphism $\psi' : U \to U^* \otimes K_S \otimes p^* K_\Sigma^{-1}$. Setting $V = p_*(U)$ we have that $\psi'$ descends to a symmetric anti-linear isomorphism $V \to V^*$ covering $f$ and hence defines an equivariant vector bundle on $\Sigma$. Sending $W$ to $p_*(W \otimes K_S^{1/2} \otimes p^* K_\Sigma^{-1/2} )$ defines a push-forward homomorphism $p_* : K^0_{\mathbb{Z}_2}(S) \to K^0_{\mathbb{Z}_2}(\Sigma)$. When $f$ has fixed points the following theorem completely determines the push-forward in equivariant $K$-theory:
\begin{theorem}\label{prop:eqpush}
Let $S^1 \subset \Sigma$ be a fixed component of $f$ and $m^+,m^-$ the dimensions of the $\pm 1$-eigenspaces of the $\mathbb{Z}_2$-action on $V$ over $S^1$. If $\tilde{f}$ has no fixed points then $m^+ = m^-$. If $\tilde{f}$ has fixed points then $K_\Sigma^{1/2},K_S^{1/2}$ are real with respect to $f,\tilde{f}$. Let $S^1_1, \dots, S^1_k \subset S$ be the fixed circle components of $\tilde{f}$ lying over $S^1$ and $m_i^+,m_i^-$ the dimensions of the $\pm 1$-eigenspaces of the $\mathbb{Z}_2$-action on $W$ over $S^1_i$. Let $d_i$ be the degree of $p : S_i^1 \to S^1$ using the orientations induced by $\gamma_S,\gamma_\Sigma$. Then
\begin{equation*}
m^+ - m^- = \sum_{i=1}^k (m^+_i  -m^-_i) d_i.
\end{equation*}
\end{theorem}
\begin{proof}
Let $x \in S^1$ be a point of $S^1$ which is not a branch point. Consider the points in the inverse image $p^{-1}(x)$. Pairs of points in $p^{-1}(x)$ exchanged by $\tilde{f}$ will push down to give a $\mathbb{Z}_2$-action with $\pm 1$-eigenspaces of equal dimension, so will not contribute to $m^+ - m^-$. Therefore we need only consider points in $p^{-1}(x)$ which are fixed by $\tilde{f}$. Let $x_i^1, \dots , x_i^{r_i}$ be the points of $p^{-1}(x)$ lying on $S^1_i$. Let $\xi_\Sigma$ be a real non-vanishing section of $K_\Sigma$ over $S^1$ compatible with the orientation of $S^1$ determined by $\gamma_\Sigma$ and similarly let $\xi_i^j$ be a real non-vanishing section of $K_S$ over $S^1_i$ compatible with $\gamma_S$. Let $\epsilon_i^j = \pm 1$ according to whether $p^*(\xi_\Sigma)$ is a positive or negative multiple of $\xi_i^j$, so $d_i = \epsilon_i^1 + \dots + \epsilon_i^{r_i}$. Let $\varphi : W \to W$ be the $\mathbb{Z}_2$-action, $h : W \to W$ a compatible hermitian metric and $\psi = h \circ \varphi : W \to W^*$. Then $U = W \otimes K_S^{1/2} \otimes p^* K_\Sigma^{-1/2}$ and we obtain a map $\widetilde{\psi} = \psi \otimes \gamma_S \otimes \gamma_\Sigma^{-1} : U \to U^* \otimes K_S \otimes p^* K^{-1}_\Sigma$. Choose $e_\Sigma \in (K_\Sigma^{1/2})_x$ such that $e^2_\Sigma = \xi_\Sigma(x)$ and $e_i^j \in (K_S^{1/2})_{x_i^j}$ with $(e_i^j)^2 = \xi_i^j(x_i^j)$. Now let $\alpha,\beta \in U_{x_i^j}$ and write $\alpha = \alpha' \otimes e_i^j \otimes e_\Sigma^{-1}$, $\beta = \beta' \otimes e_i^j \otimes e_\Sigma^{-1}$ with $\alpha',\beta' \in W_{x_i^j}$. Then, as $\widetilde{\psi}$ is defined through relative duality, we have $\widetilde{\psi}(\alpha) \beta/ dp = \psi(\alpha')\beta' \otimes \xi_i^j \otimes \xi_\Sigma^{-1} / dp$. If $p^*(\xi_\Sigma) = \rho_i^j \xi_i^j$, this is $(\rho_i^j)^{-1} \psi(\alpha')\beta'$. This contributes a sign of $\epsilon_i^j( m_i^+ - m_i^-)$ to $m^+ - m^-$. Summing we obtain the result.
\end{proof}

For the cases of real and quaternionic $K$-theory suppose $W$ is a holomorphic bundle on $S$ with real or quaternionic structure $\varphi : W \to W$. To define the push-forward we require $\varphi$ to be holomorphic meaning the map $\alpha \mapsto \varphi \circ \alpha \circ \tilde{f}$ send holomorphic germs to holomorphic germs. Since $K_S^{1/2} , K_\Sigma^{1/2}$ carry real or quaternionic structures it follows that the vector bundle $U = W \otimes K_S^{1/2} \otimes p^*K_\Sigma^{-1/2}$ is similarly real or quaternionic and this structure descends to $V = p_*(U)$. If $K_S^{1/2},K_\Sigma^{1/2}$ are both real or both quaternionic we obtain homomorphisms $p_* : KR^0(S) \to KR^0(\Sigma)$ and $p_* : KH^0(S) \to KH^0(\Sigma)$. On the other hand if one of $K_S^{1/2},K_\Sigma^{1/2}$ is real and the other quaternionic, we obtain push-forwards of the form $p_* : KR^0(S) \to KH^0(\Sigma)$ and $p_* : KR^0(S) \to KH^0(\Sigma)$. The following Theorem, which can be proved similarly to Theorem \ref{prop:eqpush}, completely determines the push-forward in $KR$-theory:
\begin{theorem}\label{prop:krpush}
Suppose $W$ has a real structure with respect to $\tilde{f}$. Let $S^1 \subset \Sigma$ be a fixed component of $f$ and $\omega$ the first Stiefel-Whitney class of $V$ over $S^1$. If $\tilde{f}$ has no fixed points then $\omega = 0$. If $\tilde{f}$ has fixed points then $K_S,K_\Sigma$ have real square roots. Choose real square roots $K_S^{1/2},K_\Sigma^{1/2}$ such that their restriction over any fixed component of $\tilde{f},f$ are trivial as real bundles. Let $S^1_1, \dots, S^1_k \subset S$ be the fixed circle components of $\tilde{f}$ lying over $S^1$ and $\omega_i$ the first Stiefel-Whitney class of $W$ over $S^1_i$. Then
\begin{equation*}
\omega = \sum_{i=1}^k \omega_i \; ( {\rm mod} \, 2).
\end{equation*}
\end{theorem}

In the case where we start with a quaternionic bundle on $S$ the push-forward follows from Grothendieck-Riemann-Roch.

%%%%%%%%%%%%%%%%%%%%%%%%%%%%%%%%%%%%%%%%%%%%%%%%%%%%%%%

\subsection{Spectral data and $K$-theory}

Let $(V,\Phi)$ be a Higgs bundle with $GL(m,\mathbb{R})$-holonomy and $L$ the corresponding spectral line bundle, hence $V = p_*( L \otimes K^{(m-1)/2})$. By Theorem \ref{proporder2} we have $L \simeq L^*$, so $L^2$ is trivial and $L$ is a real line bundle on $S$. Then $L$ defines a class $[L] \in KO^0(S)$ and we see that the $KO$-theory class of $V$ is given by the push-forward $[V] = p_*[L]$. This observation was recently made by Hitchin \cite{Char} and was used to relate the second Stiefel-Whitney class of $V$ to the mod $2$ index of $L$. Note that the fixed points of $i_1$ with holonomy in $GL(m/2,\mathbb{H})$ always have singular spectral curves \cite{hitchin13}, so on restricting to smooth fibres we do not see these fixed points.\\

Next consider fixed points of the involution $i_2$. Let $h : \mathcal{M}_{GL(m,\mathbb{C})} \to \mathcal{A}_{GL(m,\mathbb{C})}$ be the Hitchin fibration, where for $GL(m,\mathbb{C})$ we have $\mathcal{A}_{GL(m,\mathbb{C})} = \bigoplus_{i=1}^m H^0(\Sigma , K^i)$. Let $f : \mathcal{A}_{GL(m,\mathbb{C})} \to \mathcal{A}_{GL(m,\mathbb{C})}$ be the map sending a holomorphic differential $\omega$ to $f^*(\overline{\omega})$. We show in \cite{BS13} that $h \circ i_2 = f \circ h$. Therefore the fixed points of $i_2$ lie over fixed points of $f$. Let $p : S \to \Sigma$ be a non-singular spectral curve associated to a fixed point $a \in \mathcal{A}_{GL(m,\mathbb{C})}$ of $f : \mathcal{A}_{GL(m,\mathbb{C})} \to \mathcal{A}_{GL(m,\mathbb{C})}$. We extend $f$ to an involution $\tilde{f} : K \to K$ on the total space of $K$ by setting $\tilde{f}( x )  = f^*(\overline{x})$ for $x \in K$. Then since $S$ corresponds to a fixed point $a \in \mathcal{A}_{GL(m,\mathbb{C})}$, we find that $\tilde{f}$ restricts to an anti-holomorphic involution on $S$ covering $f$. Choosing a spin structure $K^{1/2}$, we identify the fibre of the Hitchin map over $a$ with the Picard variety $Pic(S)$. Assume that the chosen spin structure $K^{1/2}$ is preserved by $f$, so we obtain a real or quaternionic structure $\gamma : K^{1/2} \to K^{1/2}$. The action of $i_2$ on the fibre $Pic(S)$ is then given by $L \mapsto \tilde{f}^*( \overline{L}^* )$ \cite{BS13}. Choosing a hermitian structure on $L$ which is $\tilde{f}$-invariant, we have $L \simeq \tilde{f}^*(L)$, so the involution $\tilde{f}$ lifts to $L$ defining a class $[L] \in K^0_{\mathbb{Z}_2}(S)$. The class of $V$ in $K^0_{\mathbb{Z}_2}(\Sigma)$ is then given by the push-forward $[V] = p_*[L]$. As discussed previously the class of $[V]$ in equivariant $K$-theory is only well-defined modulo replacing the involution $\varphi : V \to V$ with $-\varphi$. This corresponds to the ambiguity in lifting $\tilde{f}$ to an involution on $L$. 

Using Theorem \ref{prop:eqpush} we can determine which classes in $K^0_{\mathbb{Z}_2}(\Sigma)$ lie in the image of the push-forward. Note that the push-forward depends on the topology of $S,\Sigma$ and the maps $\tilde{f},f$, but not the complex structures. The classes in the image of $p_*$ are the isomorphism classes of equivariant bundles which admit an equivariant flat connection with reductive holonomy and smooth spectral curve. The condition $[V] \in {\rm im}(p_*)$ may be interpreted as an analogue of the Milnor-Wood inequality for equivariant flat connections.\\

Finally, consider fixed points of the involution $i_3$. For a Higgs bundle $(V,\Phi)$ we have $i_3(V,\Phi) = (f^*(\overline{V}) , f^*(\overline{\Phi}))$. It follows that $h \circ i_2 = f \circ h$, so that fixed points of $i_3$ must lie over fixed points of $f : \mathcal{A}_{GL(m,\mathbb{C})} \to \mathcal{A}_{GL(m,\mathbb{C})}$. Let $a \in \mathcal{A}_{GL(m,\mathbb{C})}$ be a fixed point corresponding to a smooth spectral curve $p : S \to \Sigma$. As before we have an anti-holomorphic involution $\tilde{f} : S \to S$. Choosing an $f$-invariant spin structure $K^{1/2}$ we have that the action of $i_3$ on the fibre $Jac(S)$ is given by $L \mapsto \tilde{f}^*(\overline{L})$. Fixed points correspond to line bundles with $L \simeq f^*(\overline{L})$. Thus $L$ carries either a real or quaternionic structure which pushes down to a real or quaternionic structure on $V = p_*(L \otimes K^{(m-1)/2})$. The class of $V$ in $KR$- or $KH$-theory is then the push-forward $[V] = p_*[L]$. 

\begin{remark}
Given a spin structure $K^{1/2}$, recall that in \cite{Hit92} Hitchin constructed a section $s : \mathcal{A}_{GL(m,\mathbb{C})} \to \mathcal{M}_{GL(m,\mathbb{C})}$ of the Hitchin fibration which is invariant under $i_1$. Choosing $K^{1/2}$ with $f^*( \overline{K}^{1/2} ) \simeq K^{1/2}$, we have that $s$ is fixed by all three involutions $i_1,i_2,i_3$. This shows that the fixed point sets of $i_1,i_2,i_3$ are non-empty and contain smooth points.
\end{remark}

Using Theorem \ref{prop:krpush} we may determine which classes in $KR$-theory lie in the image of the push-forward. This gives a constraint on the topology of a real bundle $[V]$ to admit a real flat connection with reductive holonomy and smooth spectral curve. In this case the constraint for $[V]$ to lie in the image of $p_*$ is simple to state: for each fixed component $S^1 \subset \Sigma$ of $f$ for which $p^{-1}(S^1)$ contains no fixed point of $\tilde{f}$, we require $V|_{S^1}$ to be trivial as a real bundle.

%%%%%%%%%%%%%%%%%%%%%%%%%%%%%%%%%%%%%%%%%%%%%%%%%%%%%%%%%%%%%%%%%%

\section{Fixed points as representations}\label{secrep}

In this section we give a representation theoretic description of the fixed point set of the involutions for any complex semisimple group $G$. The case of reductive groups is similar but one must replace the fundamental group by a central extension.

%%%%%%%%%%%%%%%%%%%%%%%%%%%%%%%%%%%%%%%%%%%%%%%%%%%%%%%%%%%%%%%%%

\subsection{Fixed points of $i_1$ and real holonomy}

A representation $\phi : \pi_1(\Sigma) \to G$ is called reductive if the induced action of $\pi_1(\Sigma)$ on $\mathfrak{g}$ by the adjoint representation decomposes into a direct sum of irreducible representations. Let ${\rm Hom}^+(\pi_1(\Sigma) , G)$ be the set of all reductive representations of $\pi_1(\Sigma)$ into $G$ given the compact-open topology. The group $G$ acts on this space by conjugation and the quotient ${\rm Hom}^+(\pi_1(\Sigma) , G)/G$ is Hausdorff \cite{ric}. As discussed in Section \ref{sec:hyper} there is a homeomorphism between the moduli space of polystable Higgs bundles and the character variety ${\rm Hom}^+( \pi_1(\Sigma) , G)/G$ of reductive representations. The correspondence sends a solution $(\overline{\partial}_A , \Phi )$ of the Hitchin equations to the monodromy representation of the associated connection $\nabla = \nabla_A + \Phi + \Phi^*$. We say that a reductive representation $\phi : \pi_1(\Sigma) \to G$ is simple if for all $g \in G$ such that $Ad_g \circ \phi = \phi$, we have $g \in Z(G)$.\\

Consider a basepoint $x_0 \in \Sigma$ and write $\pi_1(\Sigma)$ for $\pi_1(\Sigma , x_0)$. Let $\gamma$ be a path from $x_0$ to $f(x_0)$ and define $f_* : \pi_1(\Sigma) \to \pi_1(\Sigma)$ by $f_*(p) = \gamma . f(p) . \gamma^{-1}$. For $\mu = [\gamma . f(\gamma)] \in \pi_1(\Sigma)$ one has  $f_*^2(p) = Ad_\mu(p)$ and $f_*(\mu) = \mu$. In terms of conjugacy classes of representations $\phi : \pi_1(\Sigma) \to G$, the involutions $i_1,i_2,i_3$ take the form
\begin{equation*}
\begin{aligned}
i_1(\phi) & = \sigma \circ \phi, \\
i_2(\phi) &= \phi \circ f_*, \\
i_3(\phi) &= \sigma \circ \phi \circ f_*.
\end{aligned}
\end{equation*}
This description of the involutions makes the relation $i_1 i_2 = i_3$ especially clear.\\

Let $\sigma,\sigma'$ be anti-involutions of $G$. Following \cite{R12}, we say that $\sigma, \sigma'$ are inner equivalent if $\sigma' = Ad_h \circ \sigma$ for some $h \in G$. Note that it is possible for distinct real forms to be inner equivalent.
\begin{proposition}
Let $\phi : \pi_1 \to G$ be a simple fixed point of $i_1$. Then there exists an anti-involution $\sigma'$ inner equivalent to $\sigma$ such that $\sigma' \circ \phi = \phi$. Conversely if $\sigma'$ is inner equivalent to $\sigma$ and $\phi$ is a reductive representation with $\sigma' \circ \phi = \phi$, then $\phi$ is a fixed point of $i_1$.
\end{proposition}
\begin{proof}
If $\phi$ is a fixed point of $i_1$ then $\sigma \circ \phi = Ad_g \circ \phi$ for some $g \in G$. Applying $\sigma$ twice and using simplicity gives $\sigma(g)g \in Z(G)$. Let $\sigma' = Ad_{g^{-1}} \circ \sigma$. Then $\sigma'$ is an anti-holomorphic involution inner equivalent to $\sigma$. Clearly we have $\sigma' \circ \phi = \phi$. The proof of the converse is immediate.
\end{proof}

\begin{remark}
For example if we consider $GL(m,\mathbb{R}) \subset GL(m,\mathbb{C})$ then aside from $GL(m,\mathbb{R})$ itself the only inner equivalent real form is $GL(m/2,\mathbb{H})$ which occurs when $m$ is even.
\end{remark}

\begin{remark}
Properties of inner equivalent involutions in relation with Higgs bundles appear in \cite{R12} and are also currently being studied in \cite{current}.
\end{remark}
%%%%%%%%%%%%%%%%%%%%%%%%%%%%%%%%%%%%%%%%%%%%%%%%%%%

\subsection{Fixed points of $i_2,i_3$ and orbifold representations}

The fixed points of $i_2,i_3$ can be described in terms of the orbifold fundamental group $\pi_1^{\rm orb}(\Sigma)$ of $\Sigma$. This group is related to the usual fundamental group by a short exact sequence
\begin{equation}\label{equofg}
1 \to \pi_1(\Sigma) \to \pi_1^{\rm orb}(\Sigma) \buildrel \nu \over \to \mathbb{Z}_2 \to 1.
\end{equation}
To describe this group we consider separately the cases where $f$ has fixed points or not. If $f$ acts freely then the quotient $\Sigma' = \Sigma/f$ is a compact non-orientable surface. The orbifold fundamental group is simply the fundamental group of $\Sigma'$ and (\ref{equofg}) becomes the usual exact sequence for a double cover. When $f$ has fixed points, we take our basepoint $x_0$ to be a fixed point and $\gamma$ to be the constant path. In this case $f_*$ is an involution and $\pi_1^{{\rm orb}}(\Sigma)$ is the semi-direct product $\pi_1^{\rm orb}(\Sigma) = \mathbb{Z}_2 \ltimes \pi_1(\Sigma)$, where $\mathbb{Z}_2$ acts on $\pi_1(\Sigma)$ by $f_*$. In either case we may take the orbifold fundamental group to be $\mathbb{Z}_2 \times \pi_1(\Sigma)$ with product given by $(0,x)(0,y) = (0,xy)$, $(0,x)(1,e) = (1,x)$, $(1,e)(0,x) = (1,f_*(x))$, $(1,e)(1,e) = (0,\mu)$.\\

Let $c : \mathbb{Z}_2 \times \mathbb{Z}_2 \to \mathbb{Z}$ be given by $c(1,1) = 1$, and  $c(u,v) = 0$ otherwise. Then $c$ is a $2$-cocycle representing the non-trivial class in $H^2(\mathbb{Z}_2 , \mathbb{Z}) = \mathbb{Z}_2$. Let $\widetilde{\pi}_1^{\rm orb}(\Sigma)$ be the central extension of $\pi_1^{\rm orb}(\Sigma)$ by $\mathbb{Z}$ corresponding to $\nu^*(c) \in H^2( \pi_1^{\rm orb}(\Sigma) , \mathbb{Z})$. Observe that the extension class $\nu^*(c) \in H^2(\pi_1^{\rm orb}(\Sigma) , \mathbb{Z})$ is trivial when restricted to $\pi_1(\Sigma)$ so we obtain a short exact sequence
\begin{equation*}
1 \to \pi_1(\Sigma) \times \mathbb{Z} \to \widetilde{\pi}_1^{\rm orb}(\Sigma) \buildrel \tilde{\nu} \over \to \mathbb{Z}_2 \to 1.
\end{equation*}

Explicitly $\widetilde{\pi}_1^{\rm orb}(\Sigma)$ is the group generated by $\pi_1(\Sigma) \times \mathbb{Z}$ together with an element $y$ modulo the relations $yx = f_*(x)y$, $yz = zy$, $y^2 = z\mu$, where $x \in \pi_1(\Sigma)$ and $z$ is a generator of $\mathbb{Z}$.

\begin{proposition}
There is a bijection between conjugacy classes of representations \newline $\psi:\widetilde{\pi}_1^{\rm orb}(\Sigma) \to G$ and equivalence classes of pairs $(\phi , u)$, where $\phi : \pi_1(\Sigma) \to G$ is a representation, and $u \in G$ such that $\phi \circ f_* = Ad_{u} \circ \phi$. Two pairs $(\phi , u),(\phi',u')$ are equivalent if $\phi' = Ad_h \circ \phi$ and   $u' = Ad_h(u)$ for some $h \in G$.
\end{proposition}
\begin{proof}
Given $\psi : \widetilde{\pi}_1^{\rm orb}(\Sigma) \to G$, we let $\phi$ be the restriction of $\psi$ to $\pi_1(\Sigma) \subset \widetilde{\pi}_1^{\rm orb}(\Sigma)$ and take $u = \psi(y)$. This gives the desired bijection.
\end{proof}
The above  establishes that fixed points of $i_2$ correspond to representations of $\widetilde{\pi}_1^{\rm orb}(\Sigma)$. On the other hand there is an obstruction to representing a fixed point of $i_2$ as a representation of $\pi_1^{\rm orb}(\Sigma)$. Suppose that $\phi$ is a simple point fixed by $i_2$. Then $\phi \circ f_* = Ad_u \circ \phi$ for some $u \in G$. Since $\phi$ is simple any two such $u$ differ by an element of $Z(G)$. Applying $f_*$ twice and using simplicity we find $c = u^2 \phi(\mu)^{-1} \in Z(G)$. In order to extend $\phi$ to a representation of $\pi_1^{\rm orb}(\Sigma)$ we need to find such a $u$ with $u^2 \phi(\mu)^{-1} = e$. Replacing $u$ by $uv$ for $v \in Z(G)$ we get $(uv)^2 \phi(\mu)^{-1} = cv^2$. Hence the obstruction to extending $\phi$ to an orbifold representation is a class in $Z(G)/2Z(G)$.\\

The case of fixed points of $i_3$ is similar. For this let $\widetilde{G}$ be the semi-direct product $\mathbb{Z}_2 \ltimes G$, where $\mathbb{Z}_2$ acts on $G$ by $\sigma$. Let $\pi : \widetilde{G} \to \mathbb{Z}_2$ be the projection.

\begin{proposition}
There is a bijection between conjugacy classes of representations \newline $\psi : \widetilde{\pi}_1^{\rm orb}(\Sigma) \to \widetilde{G}$ such that $\pi \circ \psi = \tilde{\nu}$ and equivalence classes of pairs $(\phi , u)$, where $\phi : \pi_1(\Sigma) \to G$ is a representation, and $u \in G$ such that $\sigma \circ \phi \circ f_* = Ad_{\sigma(u)} \circ \phi$. Two pairs $(\phi , u),(\phi',u')$ are equivalent if $\phi' = Ad_h \circ \phi$ and  $u' = h u \sigma(h)^{-1}$ for some $h \in G$.
\end{proposition}

It follows that every fixed point of $i_3$ extends to a representation of $\widetilde{\pi}_1^{\rm orb}(\Sigma)$. As in the case of $i_2$ we can consider the problem of extending a fixed point of $i_3$ to a representation of $\pi_1^{\rm orb}(\Sigma)$. Suppose $\phi$ is a simple fixed point of $i_3$, so $\sigma \circ \phi \circ f_* = Ad_{\sigma(u)} \circ \phi$ for some $u \in G$. Applying the involutions $\sigma$ and $f_*$ twice and using simplicity we obtain an element $c = u\sigma(u) \phi(\mu)^{-1} \in Z(G)$ such that $\sigma(c) = c$. Replacing $u$ by $uv$ for $v \in Z(G)$ we have $c \mapsto c v \sigma(v)$, so we obtain a class in $Z' = \{ c \in Z(G) \, | \, c = \sigma(c) \}/ \{ v \sigma(v) \, | \, v \in Z(G) \}$. This class is the obstruction to lifting $\phi$ to a representation $\widetilde{\phi} : \pi_1^{\rm orb}(\Sigma) \to \widetilde{G}$ with $\pi \circ \widetilde{\phi} = \nu$. 

For example when $G = GL(m,\mathbb{C})$ and $\sigma$ is conjugation we find $Z' = \mathbb{Z}_2$. In this case the trivial class in $Z'$ corresponds to real bundles and the non-trivial class to quaternionic bundles.

%%%%%%%%%%%%%%%%%%%%%%%%%%%%%%%%%%%%%%%%%%%%%%%%%%%%%%%%%%%%%%%%%%%

\section{Duality}\label{sec:duality}

Let $^L G$ be the Langlands dual group of $G$. There is a correspondence between invariant polynomials for $G$ and $^L G$ giving an identification $\mathcal{A}_G \simeq \mathcal{A}_{^L G}$. The moduli spaces $\mathcal{M}_G,\mathcal{M}_{^L G}$ are then torus fibrations over a common base and their non-singular fibres are dual abelian varieties \cite{dopa}. Kapustin and Witten give a physical interpretation of this in terms of S-duality, using it as the basis for their approach to the geometric Langlands program \cite{Kap}. In this approach a crucial role is played by the various types of branes and their transformation under mirror symmetry. This duality exchanges branes according to $(B,B,B) \leftrightarrow (B,A,A)$, $(A,B,A) \leftrightarrow (A,B,A)$, $(A,A,B) \leftrightarrow (A,A,B)$. We consider here the question of how this duality acts on the fixed point sets of the real structures $i_1,i_2,i_3$. We make some conjectures without attempting to give rigorous justifications. We let $\mathcal{B}_G^a \subset \mathcal{M}_G$ denote the fixed point set of $i_a$ for $a = 1,2,3$.\\

The simplest case is the fixed point set of $i_2$, which is of type $(A,B,A)$. Since the definition of $i_2$ requires only the choice of anti-holomorphic involution $f$, we have a corresponding involution $\hat{i}_2$ on $\mathcal{M}_{^L G}$ with fixed point set an $(A,B,A)$-brane $\widehat{\mathcal{B}}^2_{^L G} \subset \mathcal{M}_{^L G}$. We conjecture that $\widehat{\mathcal{B}}^2_{^L G}$ is the dual brane to $\mathcal{B}^2_G$. We give some evidence for this in \cite{BS13}.\\

Consider next the $(A,A,B)$-brane $\mathcal{B}_G^3 \subset \mathcal{M}_G$. Since the dual brane $\widehat{\mathcal{B}}^3_{^L G} \subset \mathcal{M}_{^L G}$ is also of type $(A,A,B)$ one might conjecture that it is the fixed point set of a corresponding involution $\hat{i}_3$. To define $\hat{i}_3$ we need to choose a real structure $\hat{\sigma}$ on $^L G$. If $G$ is a simple group not of type $B_n,C_n$, the Lie algebras of $G$ and $^L G$ coincide and we have a natural choice for $\hat{\sigma}$. The $B_n,C_n$ cases however remain a mystery.\\

The most interesting case is the $(B,A,A)$-brane $\mathcal{B}_G^1 \subset \mathcal{M}_G$. The dual $\widehat{\mathcal{B}}^1_{^L G} \subset \mathcal{M}_{^L G}$ must be of type $(B,B,B)$, a submanifold which is complex with respect to $I,J,K$. One natural way of constructing $(B,B,B)$-branes in $\mathcal{M}_{^L G}$ is to take a complex subgroup $H \subset {^L G}$ and to let $\widehat{\mathcal{B}}_{^L G}^1$ be the space of $^L G$-Higgs bundles with holonomy in $H$. It remains to find a natural choice of subgroup $H$. In \cite{nad} a correspondence between real structures on $G$ and complex subgroups of the dual group $^L G$ is given. We conjecture that the correspondence given in \cite{nad} determines the correct dual brane to $\mathcal{B}_G^1$. Some evidence for this duality in the case $U(m,m) \subset GL(2m,\mathbb{C}) \leftrightarrow Sp(2m,\mathbb{C}) \subset GL(2m,\mathbb{C})$ has been shown by Hitchin in \cite{Char}.

%%%%%%%%%%%%%%%%%%%%%%%%%%%%%%%%%%%%%%%%%%%%%%%%%%%%%%%%%%%%%%%%%%%

%\newpage


\begin{thebibliography}{}
%Author, \textit{Book title} (Publisher, place year) page numbers
% etc

\bibitem[Ati66]{A66}M. F. Atiyah, $K$-theory and reality. {\em Quart. J. Math. Oxford Ser.} (2) {\bf 17} (1966) 367-386.

\bibitem[Ati71]{A71} M. Atiyah. Riemann surfaces and spin structures. {\em Ann. Sci. \'Ecole Norm. Sup.} (4) (1971) vol. 4, 47-62.

\bibitem[BarSch13]{BS13} D.~Baraglia, L.P.~Schaposnik, Higgs bundles and $(A, B, A)$-branes, arXiv:1305.4638 (2013).

 

\bibitem[BiGo08]{biswas} I. Biswas, T.L. G\'omez, 
\textit{Connections and Higgs fields on a principal bundle}, Ann. Global Anal. Geom. \textbf{33 1}, (2008) 19-46.

\bibitem[BHH10]{bhh}I. Biswas, J. Huisman, J. Hurtubise, The moduli space of stable vector bundles over a real algebraic curve. {\em Math. Ann.} {\bf 347} (2010), no. 1, 201-233.

\bibitem[BGH12]{biswas2}I. Biswas, O. Garc\'ia-Prada, J, Hurtubise, Pseudo-real principal Higgs bundles on compact Kaehler manifolds, arXiv:1209.5814 (2012).


\bibitem[BNR89]{bobi} A. Beauville, M. S. Narasimhan, S. Ramanan, 
\textit{ Spectral curves and the generalised theta divisor}, 
J. Reine Angew. Math. \textbf{398} (1989), 169-179.

\bibitem[BGM03]{bgm}S. B. Bradlow, O. Garc\'ia-Prada, I. Mundet i Riera, Relative Hitchin-Kobayashi correspondences for principal pairs. {\em Q. J. Math.} {\bf 54} (2003), no. 2, 171-208.


\bibitem[Cor88]{cor}  K. Corlette,
\textit{Flat $G$-bundles with canonical metrics}, 
J. Differential Geom. {\bf 28} (1988), 361-382.

\bibitem[DoPa12]{dopa} R. Donagi, T. Pantev, \textit{Langlands duality for Hitchin systems}, Invent. math. {\bf 189} (2012), no. 3, 653-735.

\bibitem[D87]{donald}  S.K. Donaldson, 
\textit{Twisted harmonic maps and the self-duality equations}, 
Proc. London Math. Soc. (3) \textbf{55 1}, (1987)127-131.

\bibitem[Dup69]{dup}J. L. Dupont, Symplectic bundles and $KR$-theory. {\em Math. Scand.} {\bf 24} (1969) 27-30.

\bibitem[G-P07]{R12} O. Garc\'ia-Prada, Involutions of the moduli space of $SL(n,\mathbb{C})$-Higgs bundles and real forms, 
In ``Vector Bundles and Low Codimensional Subvarieties: State of the Art and Recent Developments'' {\em Quaderni di Matematica}, Editors: G. Casnati, F.
Catanese and R. Notari (2007).

\bibitem[G-PR]{current} O.~Garc\'ia-Prada and  S.~Ramanan, \textit{Involutions of the moduli space of Higgs bundles}, in preparation.

\bibitem[He01]{helga} S. Helgason,\textit{ Differential geometry, Lie groups, and symmetric spaces}, Graduate Studies in Mathematics, AMS \textbf{34} (2001).


\bibitem[Hit87]{N1} N.J.~Hitchin,
\textit{The self-duality equations on a Riemann surface}, 
Proc. London Math. Soc. {\bf 55 3}, (1987), 59-126.

\bibitem[Hit87a]{N2} N.J.~Hitchin, 
\textit{Stable bundles and integrable systems},  
Duke Math. J. \textbf{54 1}, (1987), 91-114.

\bibitem[Hit92]{Hit92} N.J.~ Hitchin,  
\textit{Lie Groups and Teichm\"{u}ller Space}, 
Topology {\bf 31 3},  (1992), 449-473.
 

\bibitem[Hit13]{Char} N. J. Hitchin, Higgs bundles and characteristic classes, arXiv:1308.4603 (2013).

\bibitem[HitSch13]{hitchin13}  N.J.~Hitchin, L.P. Schaposnik, \textit{Nonabelianization of Higgs bundles}, arXiv:1307.0960, (2013).
 

\bibitem[KW07]{Kap} A. Kapustin, E. Witten,
\textit{Electric-magnetic duality and the geometric Langlands program}. 
Commun. Number Theory Phys.  \textbf{1}, (2007) 1-236.

\bibitem[KarWei03]{kar}M. Karoubi, C. Weibel, Algebraic and Real $K$-theory of real varieties. {\em Topology} {\bf 42} (2003), no. 4, 715-742.

 

\bibitem[Nad05]{nad} D. Nadler, Perverse sheaves on real loop Grassmannians. {\em Invent. Math.} {\bf 159} (2005), no. 1, 1-73.


\bibitem[N91]{nit}
N. Nitsure, \textit{Moduli spaces of semistable pairs on a curve}, Proc. London Math. Soc. {\bf 62} (1991). 275-300.

 

\bibitem[Ric88]{ric}R. W. Richardson, Conjugacy classes of $n$-tuples in Lie algebras and algebraic groups. {\em Duke Math. J.} {\bf 57} (1988), no. 1, 1-35.

\bibitem[Schaf12]{florent}F. Schaffhauser, Real points of coarse moduli schemes of vector bundles on a real algebraic curve. {\em J. Symplectic Geom.} {\bf 10} (2012), no. 4, 503-534.

\bibitem[Sch13]{thesis} L.P.~Schaposnik, 
\textit{Spectral data for G-Higgs bundles} (University of Oxford, Oxford 2013).

 

\bibitem[S88]{simpson88} C.T. Simpson, 
\textit{Constructing variations of Hodge structure using Yang-Mills theory and applications to uniformisation}, 
J. Amer. Math. Soc. {\bf 1}, (1988) 867-918.

\end{thebibliography}
\end{document}